\documentclass[11pt]{amsart}
\usepackage{amsmath}
\usepackage{amsxtra}
\usepackage{amscd}
\usepackage{amsthm}
\usepackage{amsfonts}
\usepackage{amssymb}
\usepackage{eucal}
\usepackage{mathrsfs}
\usepackage{color, graphics}
\usepackage[all]{xy}
\usepackage{hyperref}

\theoremstyle{plain} 
\newcommand \reg {\mathrm {reg}}

\newcommand \Gin {\ensuremath{\mathrm{Gin}}}
\newcommand \In {\ensuremath{\mathrm{in}}}
\newcommand \Proj {\ensuremath{\mathrm{Proj}}}
\newcommand \red {\ensuremath{\mathrm{red}}}
\newcommand \mult {\ensuremath{\mathrm{mult}}}
\newcommand \p {\ensuremath{\mathbb{P}}}
\newcommand \lex {\ensuremath{\mathrm{GLex}}}
\newcommand \Tor {\ensuremath{\mathrm{Tor}}}
\newcommand \st[1] {\stackrel{#1}{\longrightarrow}}
\theoremstyle{definition}
\newtheorem{Thm}{Theorem}[section]
\newtheorem{Defn}{Definition}[section]
\newtheorem{lem}[Thm]{Lemma}
\newtheorem{remk}[Thm]{Remark}

\newtheorem{Ex}[Thm]{Example}
\newtheorem{prop}[Thm]{Proposition}
\newtheorem{Cor}[Thm]{Corollary}
\newtheorem{RQ}[Thm]{Remark and Question}

\begin{document}

\title{The Degree Complexity of Smooth Surfaces of codimension 2}
\author[J.\ Ahn, S.\ Kwak and Y.\ Song]{Jeaman Ahn , Sijong Kwak and YeongSeok Song}
\address{ Mathematics Education, Kongju National Univesity,
182 Sinkwan-dong Kongju-si Chungnam, 314-702, Korea}
\email{jeamanahn@kongju.ac.kr}

\address{Department of Mathematics, Korea Advanced Institute of Science and Technology,
373-1 Gusung-dong, Yusung-Gu, Daejeon, Korea}
\email{skwak@kaist.ac.kr}
\thanks{The second author was supported in part by the National Research Foundation of Korea funded
by the Ministry of Education, Science and Technology(grant No. 2010-0001652)}
\address{Department of Mathematics, Korea Advanced Institute of Science and Technology,
373-1 Gusung-dong, Yusung-Gu, Daejeon, Korea}
\email{lassei@kaist.ac.kr}

\begin{abstract}
For a given term order, the degree complexity of a projective scheme is defined
by the maximal degree of the reduced Gr\"{o}bner basis of its defining saturated
ideal in generic coordinates~\cite{BM}.
It is well-known that the degree complexity with respect to the graded reverse
lexicographic order is equal to the Castelnuovo-Mumford regularity~\cite{BS}.
However, much less is known if one uses the graded lexicographic order~\cite{A},~\cite{CS}.

In this paper, we study the degree complexity of a smooth irreducible surface in $\p^4$
with respect to the graded lexicographic order and its geometric meaning.
Interestingly, this complexity is closely related to the invariants of the double curve
of a surface under a generic projection. As results, we prove that except a few cases, the degree
complexity of a smooth surface $S$ of degree $d$ with $h^0(\mathcal I_S(2))\neq 0$ in $\p^4$
is given by $2+\binom{\deg Y_1(S)-1}{2}-g(Y_1(S))$, where $Y_1(S)$ is a double curve
of degree $\binom{d-1}{2}-g(S \cap H)$ under a generic projection of $S$.
In particular, this complexity is actually obtained at the monomial
$$
x_0 x_1 x_3\,^{\binom{\deg Y_1-1}{2}-g(Y_1(S))}
$$
where $k[x_0,x_1, x_2,x_3,x_4]$ is a polynomial ring defining $\p^4$.
Exceptional cases are a rational normal scroll, a complete intersection surface of $(2,2)$-type,
or a Castelnuovo surface of degree $5$ in $\p^4$ whose degree complexities are in fact equal to their
degrees. This complexity can also be expressed in terms of degrees of defining equations
of $I_S$ in the same manner as the result of A. Conca and J. Sidman~\cite{CS}.
We also provide some illuminating examples of our results
via calculations done with {\it Macaulay 2}~\cite{GS}.
\end{abstract}

\maketitle

\tableofcontents \setcounter{page}{1}
\section{Introduction}\label{section_1}
D. Bayer and D. Mumford in \cite{BM} have introduced the degree complexity of a homogeneous ideal $I$
with respect to a given term order $\tau$ as the maximal degree of the reduced Gr\"{o}bner basis of $I$,
and this is exactly the highest degree of minimal generators of the initial ideal of $I$.
Even though degree complexity depends on the choice of coordinates, it is constant in generic coordinates
since the initial ideal of $I$ is invariant under a generic change of coordinates, which is the so-called
the generic initial ideal of $I$ \cite{E1}.

For the graded lexicographic order (resp. the graded reverse lexicographic order), we denote by $M(I)$
(resp. $m(I)$) the degree complexity of $I$ {\it in generic coordinates}.
For a projective scheme $X$, the degree complexity of $X$ can also be defined as $M(I_X)$ (resp. $m(I_X)$)
for the graded lexicographic order (resp. the graded reverse lexicographic order) where $I_X$ is the defining
saturated ideal of $X$.

D. Bayer and M. Stillman have shown in ~\cite{BS} that $m(I)$ is exactly
equal to the Castelnuovo-Mumford regularity $\reg(I)$.
Then what can we say about $M(I)$?
A. Conca and J. Sidman proved in ~\cite{CS} that if $I_C$ is the defining ideal
of a smooth irreducible complete intersection curve $C$ of type $(a,b)$ in $\p^3$ then
$M(I_C)$ is $1 + \frac{ab(a-1)(b-1)}{2}$ with the exception of the case $a=b=2$, where
$M(I_C)$ is $4$.
Recently, J. Ahn has shown in ~\cite{A} that if $I_C$ is the defining ideal of a
non-degenerate smooth integral curve of degree $d$ and genus $g(C)$ in $\p^r$ (for $r\ge 3$),
then $M(I_C) = 1 + \binom{d-1}{2}-g(C)$
with two exceptional cases.

In this paper, we would like to compute the degree complexity of a smooth surface $S$
in $\p^4$ with respect to the graded lexicographic order.
Interestingly, this complexity is closely related to the invariants of the double curve of $S$
under the generic projection. Our main results are: if $S \subset \p^4$ is a smooth irreducible
surface of degree $d$ with $h^0(\mathcal I_S(2))\neq 0$, then the degree complexity $M(I_S)$ of
$S$ is given by $2+\binom{\deg Y_1(S)-1}{2}-g(Y_{1}(S))$ with three
exceptional cases, where
$Y_1(S)$ is a smooth double curve of $S$ in $\p^3$ under a generic projection
and $\deg Y_{1}(S)=\binom{d-1}{2}-g(S\cap H)$. Moreover, this complexity is actually obtained
at the monomial
$$
x_0 x_1 x_3\,^{\binom{\deg Y_1-1}{2}-g(Y_1(S))}
$$
where $k[x_0,x_1, x_2,x_3,x_4]$ is a polynomial ring defining $\p^4$.

On the other hand, $M(I_S)$ can also be expressed in terms of degrees of defining equations
of $I_S$ in the same manner as the result of A. Conca and J. Sidman \cite {CS} (see Theorem~\ref{cor:01}).
Note that if $S$ is a locally Cohen-Macaulay surface with $h^0(\mathcal I_S(2))\neq 0$ then there
are two types of $S.$ One is a complete intersection of $(2,\alpha)$-type and the other is projectively
Cohen-Macaulay of degree $2 \alpha -1.$
For those cases, $\deg Y_1(S)$, $g(Y_1(S))$ and $g(S\cap H)$ can be obtained in terms of $\alpha$.

Consequently, if $S$ is a complete intersection of $(2, \alpha)$-type for some $\alpha\geq 3$ then
$M(I_S)=\frac{1}{2}(\alpha^4-4\alpha^3+5\alpha^2-2\alpha+4).$ If $S$ is projectively Cohen-Macaulay
of degree $2\alpha-1,$ $\alpha \geq 4$, then
$M(I_S) =
\frac{1}{2}(\alpha^4-6\alpha^3+13\alpha^2-12\alpha+8)$~(see Theorem ~\ref{cor:01}). 
Exceptional cases are a rational normal scroll, a complete intersection surface of $(2,2)$-type, or
a Castelnuovo surface of degree $5$ in $\p^4.$ In these cases, $M(I_S)=\deg(S)$~(see Theorem ~\ref{mainthm2}). 

The main ideas are divided into two parts: one is to show that the
degree complexity $M(I_S)$ is given by the maximum of
$\reg(\Gin_{\lex}(K_{i}(I_{S})))+i$ for $i=0,1$ and the other part
is to compare the schemes of multiple loci defined by partial
elimination ideals and their classical scheme
structures defined by the Fitting ideals of an $\mathcal O_{\p^3}$-module $\pi_{*}\mathcal O_S$
where $\pi$ is a generic projection of $S$ to $\p^3$.\\

{\bf Acknowledgements} We are very grateful to the anonymous referee for valuable and helpful suggestions.
In addition, the program {\it Macaulay 2} has been useful to us in computations of concrete examples, and
in understanding the degree complexity of smooth surfaces in $\p^4$.


\section{Notations and basic facts}\label{section_2}

\begin{itemize}
 \item We work over an algebraically closed field $k$ of characteristic zero.

\item Let $R=k[x_0 , \ldots , x_r]$ be a polynomial ring over $k$. For a closed subscheme $X$ in $\p^r$,
we denote the defining saturated ideal of $X$ by
$$\label{eq1}
I_X = \underset{m=0}{\overset{\infty}{\bigoplus}}H^{0}(\mathcal{I}_X(m))
$$

\item For a homogeneous ideal $I$, the Hilbert function of $R/I$ is defined by
$H(R/I, m):=\dim_k (R/I)_m$ for any non-negative integer $m$.
We denote its corresponding Hilbert polynomial by $P_{R/I}(z)\in \Bbb Q[z]$. If $I=I_X$
then we simply write $P_X(z)$ instead of $P_{R/I_X}(z)$.

\item We write $\rho_a(X)=(-1)^{\dim(X)}(P_X(0)-1)$ for the arithmetic genus of $X$.

 \item For a homogeneous ideal $I\subset R$, consider a minimal free resolution
$$\cdots \rightarrow \bigoplus_jR(-i-j)^{\beta_{i,j}(I)}\rightarrow\cdots\rightarrow\bigoplus_jR(-j)^{\beta_{0,j}(I)}\rightarrow I\rightarrow 0$$
of $I$ as a graded $R$-modules.
We say that $I$ is $m$-regular if $\beta_{i,j}(I)=0$ for all $i\geq 0$ and $j\geq m$.
The Castelnuovo-Mumford regularity of $I$ is defined by
\[\reg(I):=\min\{\,m\,\mid\, I \textup{ is } m \textup{-regular}\}.\]

\item Given a term order $\tau$, we define the initial term $\In_{\tau}(f)$ of  a homogeneous polynomial $f \in R$
to be the greatest monomial of $f$ with respect to $\tau$. If $I \subset R$ is a homogeneous ideal,
we also define the initial ideal $\In_{\tau}(I)$ to be the ideal generated by $\{ \In_{\tau}(f) \mid f \in I \}$.
A set $G = \{ g_1 , \ldots , g_n \} \subset I$ is said to be a Gr\"{o}bner basis if
\[(\In_{\tau}(g_1) , \ldots , \In_{\tau}(g_n) ) = \In_{\tau}(I).\]

\item For an element $\alpha = (\alpha_0 , \ldots , \alpha_r) \in \mathbb{N}^r$ we define
the notation $x^{\alpha} = x_{0}^{\alpha_{0}} \cdots x_{r}^{\alpha_{r}}$ for monomials. Its degree is
$\mid\alpha\mid =\sum_{i=0}^{r}\alpha_{i}.$

For two monomial terms $x^{\alpha}$ and $x^{\beta}$, the {\it graded lexicographic order} is defined by
$x^{\alpha} \geq_{\mathrm{GLex}} x^{\beta}$ if and only if $|\alpha| > |\beta|$ or $|\alpha| = |\beta|$
and if the left most nonzero entry of $\alpha - \beta$ is positive. The {\it graded reverse lexicographic order}
is defined by $x^{\alpha} \geq_{\mathrm{GRLex}} x^{\beta}$ if and only if
we have $|\alpha| > |\beta|$ or $|\alpha| = |\beta|$ and if the right most nonzero entry of $\alpha - \beta$
is negative.

\item In characteristic 0, we say that a monomial ideal $I$ has the Borel-fixed property if, for some monomial $m$,
we have $x_im\in I$, then $x_jm\in I$ for all $j\leq i$.

\item Given a homogeneous ideal $I \subset R$ and a term order $\tau$, there is a Zariski open subset
$U \subset GL_{r+1}(k)$ such that $\In_{\tau}(g(I))$ is constant. We will call $\In_{\tau}(g(I))$
for $g \in U$ the generic initial ideal of $I$ and denote it by $\Gin_{\tau}(I)$. Generic initial ideals
have the Borel-fixed property (see \cite{E1},\cite{G}).

\item For a homogeneous ideal $I\subset R$, let $m(I)$ and $M(I)$ denote the maximum of the degrees of minimal generators of $\Gin_{\mathrm{GRLex}}(I)$ and $\Gin_{\mathrm{GLex}}(I)$ respectively.

\item If $I$ is a Borel fixed monomial ideal then $\reg(I)$ is exactly the maximal degree of minimal generators of $I$ (see \cite{BS},\cite {G}). This implies that
 \[m(I)=\reg(\Gin_{\mathrm{GRLex}}(I)) \text{ and } M(I)=\reg(\Gin_{\mathrm{GLex}}(I)).\]

\end{itemize}

\bigskip


\section{Gr\"obner bases of partial elimination ideals}
\begin{Defn}\label{defpartial}
Let $I$ be a homogeneous ideal in $R$. If $f \in I_d$ has leading term $\In(f)=x_{0}^{d_{0}} \cdots x_{r}^{d_{r}}$,
we will set $d_{0}(f)=d_0$, the leading power of $x_0$ in $f$. We let
\begin{equation*}
\widetilde{K}_{i}(I)= \bigoplus_{d \geq 0} \{f \in I_{d} \mid  d_{0}(f) \leq i  \}.
\end{equation*}
If $f \in \widetilde{K}_{i}(I)$, we may write uniquely
\begin{equation*}
f=x_{0}^{i}\overline{f} + g,
\end{equation*}
where $d_{0}(g) < i$. Now we define $K_{i}(I)$ as the image of $\widetilde{K}_{i}(I)$ in $\bar{R}=k[x_1 \ldots x_r]$
under the map $f\rightarrow \overline{f}$ and we call $K_{i}(I)$ the $i$-th partial elimination ideal of $I$.
\end{Defn}

\begin{remk}
We have an inclusion of the partial elimination ideals of $I$:
$$ I\cap\bar{R}= K_0(I)\subset K_1(I)\subset \cdots \subset K_i(I)\subset  K_{i+1}(I)\subset\cdots \subset \bar{R}.$$
Note that if $I$ is in generic coordinates and $i_0=\min\{i \mid I_i \neq 0 \}$ then $K_{i}(I)=\bar R$ for all $i\geq i_0$.
\end{remk}

The following result gives the precise relationship between partial elimination ideals and the geometry of the projection map from $\p^r$
to $\p^{r-1}$. For a proof of this proposition, see \cite[Propostion 6.2]{G}.

\begin{prop}\label{defset}
Let $X \subset \p^{r}$ be a reduced closed subscheme and let $I_X$ be the defining ideal of $X$.
Suppose $p= [1,0, \ldots ,0 ] \in \p^r \setminus X$ and that $\pi : X \rightarrow \p^{r-1}$ is the projection from
the point $p \in \p^r$ to $x_{0} = 0$. Then, set-theoretically, $K_{i}(I_{X})$ is the ideal of
$\{ q \in \pi(X) \mid \mult_{q}(\pi(X))> i \}.$
\end{prop}

For each $i\geq 0$, note that we can give a scheme structure on the set
$$Y_i(X):=\{ q \in \pi(X) \mid \mult_{q}(\pi(X)) > i \} $$
from the $i$-th partial elimination ideal $K_i(I)$. Let
\begin{equation*}
Z_{i}(X):=\Proj(\bar{R}/K_{i}(I_{X})),
\end{equation*}
where $\bar{R}=k[x_1 \ldots x_r]$. Then it follows from Proposition~\ref{defset} that
\begin{equation*}
Z_{i}(X)_{\red}=Y_i(X).
\end{equation*}

\begin{remk}
Let $X\subset \p^r$ be a smooth variety of codimension two and let $\pi : X \rightarrow \p^{r-1}$
be a generic projection of $X$. A classical scheme structure on the set $Y_i(X)$ is given by $i$-th
Fitting ideal of the $\mathcal O_{\p^{r-1}}$-module $\pi_{*}\mathcal O_X$ (see \cite{KLU},\cite {MP}).
Throughout this paper, we use the notation $Y_i(X)$ in the sense that it is a closed subscheme defined
by Fitting ideal of $\pi_{*}\mathcal O_X$, as distinguished from the notation $Z_i(X)$.
We show that if $S\subset \p^4$ is a smooth surface lying in a quadric surface then $Y_1(S)$ and
$Z_1(S)$ have the same reduced scheme structure (see Theorem~\ref{thm4.1}),
which will be used in the proof of Proposition~\ref{mainthm2}.
\end{remk}

It is natural to ask: what is a Gr\"{o}bner basis of $K_i(I)$? Recall that any non-zero polyomial $f$ in $R$
can be uniquely written as $f=x^t\bar f+g$ where $d_0(g)<t$. A. Conca and J. Sidman \cite{CS} show that
if $G$ is a Gr\"{o}bner basis for an ideal $I$ then the set
$$
G_i=\{ \bar{f} \mid f\in G \mbox{ with } d_0(f)\leq i\}
$$
is a Gr\"{o}bner basis for $K_i(I)$. However if $I$ is in generic coordinates then there is a more refined
Gr\"{o}bner basis for $K_i(I)$, which plays an important role in this paper.

\begin{prop}\label{thm2}
Let $I$ be a homogeneous ideal {\it in generic coordinates} and $G$ be a Gr\"{o}bner basis for $I$ with respect to the graded lexicographic order.
Then, for each $i\geq 0$,
\begin{itemize}
\item[(a)] the $i$-th partial elimination ideal $K_i(I)$ is in generic coordinates;
\item[(b)] $G_i = \{ \bar{f} \mid f \in G ~\mbox{with} ~d_0 (f) = i\}$
is a Gr\"{o}bner basis for $K_i(I)$.
\end{itemize}

\end{prop}
\begin{proof}
(a) is in fact proved in Proposition~3.3 in \cite{CS}.
For a proof of (b),  it suffices to show that $\langle \In(G_i)\rangle = \In(K_i (I))$
by the definition of Gr\"obner bases. Since $G_i \subset K_i (I)$,
we only need to show that $\langle \In(G_i) \rangle \supset \In(K_i (I))$.
Now, we denote $\mathcal{G}(I)$ by the set of minimal generators of $I$.
Let $m \in \In(K_i (I))$ be a monomial. Then there is a monomial generator $M \in \mathcal{G}(\In(K_i (I)))$
such that $M$ divide $m$.

We claim that $x_0^i M \in \mathcal{G}(\In(I))$ if and only if $ M \in \mathcal{G}(\In(K_i (I))).$

If the claim is proved then we will be done. Indeed, for $ M \in \mathcal{G}(\In(K_i (I)))$,
we see that $x_0^iM \in \mathcal{G}(\In(I))$. This implies that there exists a polynomial
$f = x_0^i \bar{f} +g \in G$  with $d_0(g) < i$ such that
$$\In(f)=x_0^i\In(\bar f)=x_0^i M.$$
This means that $M=\In(\bar f) \in \langle \In(G_i) \rangle$. Thus we have $m \in \langle \In(G_i) \rangle$.\\
Here is a proof of the claim: suppose that $x_0^i M \in \mathcal{G}(\In(I))$ then we can say that $x_0^i M \in \In(I)$.
Thus there is a polynomial $f =x_0^i \bar{f} +g \in I$ such that $d_0(g) < i$ and $\In(f) =x_0^i \In(\bar f) = x_0^i M$.
By the definition of partial elimination ideals, we have that $\bar{f} \in K_i(I)$, which means $M \in \In(K_i(I))$.
Assume that $M \notin \mathcal{G}(\In(K_i(I)))$. Then for some monomial $N \in \mathcal{G}(\In(K_i(I)))$ such that
$N$ divide $M$.
This implies that
$$
x_0^i N \in \In(I) ~\mbox{and}~ x_0^i N \mid x_0^i M,
$$
which contradicts the fact that $x_0^i M$ is a minimal generator of $\In(I)$.
Thus $M$ is contained in $\mathcal{G}(\In(K_i(I)))$.

Conversely, suppose that there is $M \in \mathcal{G}(\In(K_i(I)))$ such that $x_0^i M \notin \mathcal{G}(\In(I))$.
Then we may choose a monomial $x_0^j N \in \mathcal{G}(\In(I))$ satisfying
\begin{equation}\label{eq:000001}
x_0 \nmid N ~\mbox{and}~x_0^j N \mid x_0^i M.
\end{equation}
Note that (\ref{eq:000001}) implies that $i \geq j \geq 0$.
Since $N \in \In(K_j(I))$ and $K_0(I) \subset K_1(I) \subset \cdots$, it is obvious that
$N \in \In(K_i(I))$ and $N$ divides $M$.
Now, we claim that $N$ can be chosen to be different from $M$. If $N = M$ then $j$ must be less than $i$.
Denote $N$ by $x_1^{j_1} \cdots x_r^{j_r}$ and choose $j_t\neq 0$. By (a), note that $K_{i}(I)$ is
in generic coordinates and so we may assume that $\In(K_i(I))$ has the Borel-fixed property.
Therefore, if we set
$N^{'}=N/x_{j_t}$ then $x_0^{j+1} N^{'} \in \In(I)$.
Replace $x_0^j N$ by $N^{''} =x_0^{j+1} N^{'}$. Then $N^{'} \in \In(K_{j+1}(I))$.
Since $j+1\leq i$, we can say that $N^{'} \in \In(K_i(I))$ and $N^{'}$ divides $M$ with $N^{'} \neq M$.
This contradicts the assumption that $M \in \mathcal{G}(\In(K_i(I)))$.
\end{proof}

\begin{remk}
The condition ``in generic coordinates" is crucial in Proposition~\ref{thm2} (b)
as the following example shows. Let $I = ( x_0 ^2, x_0 x_1, x_0x_2, x_3 )$ be a monomial ideal.
Then $G=\{ x_0 ^2, x_0 x_1, x_0x_2, x_3 \}$ is a Gr\"{o}bner basis for $I$.
Then we can easily check that
\begin{align*}
G_1&=\{ \bar{f} \mid f\in G \mbox{ with } d_0(f)\leq 1\}=(x_1, x_2, x_3),\\
G_1^{'}&=\{ \bar{f} \mid f\in G \mbox{ with } d_0(f)= 1\}= (x_1, x_2).
\end{align*}
This shows that $G_1^{'}$ is not a Gr\"obner basis for $K_1(I)$.
\end{remk}

We have the following corollary from Proposition~\ref{thm2}.
\begin{Cor}\label{cor:305}
For a homogeneous ideal $I\subset R=k[x_0,\ldots, x_r]$ {\it in generic coordinates}, we have
\[M(I)=\max \{ M(K_i(I))+i~ \mid~ 0\leq i\leq \beta\},\]
where $\beta=\min\{j ~\mid ~I_{j}\neq 0 \}$.
\end{Cor}
\begin{proof}
Note that $K_{\beta}(I)=\bar R$ for $\beta=\min\{j~\mid~ I_{j}\neq 0 \}$ by definition.
We know that $M(I)$ can be obtained from the maximal degree of generators in $\Gin(I)$.
Remember that $\mathcal{G}(I)$ is the set of minimal generators of $I$. Then by Proposition~\ref{thm2},
every generator of $\Gin(I)$ is of the form $x_0^i M$ where $M \in \mathcal{G} (\Gin(K_i (I)))$ for some $i$.
This means that $M(I) \leq M(\Gin(K_i (I)))+i$ for some $i$.
On the other hand, if for each $i$, we choose $M \in \mathcal{G}(K_i (I))$, then by Proposition~\ref{thm2},
$x_0^i M$ is contained in $\mathcal{G}(\Gin(I))$. Hence we conclude that
\[M(I)=\max \{ M(K_i(I))+i~ \mid~ 0\leq i\leq \beta\}.\]
\end{proof}

Corollary~\ref{cor:305} with the following theorem can be used to obtain the degree-complexities
of the smooth surface lying in a quadric hypersurface in $\p^4$. For a proof of this theorem, see
\cite[Theorem 4.4]{A}.

\begin{Thm}\label{curveM}
Let $C$ be a non-degenerate smooth curve of degree $d$ and genus $g(C)$ in $\p^{r}$ for some $r\ge 3$.
Then,
$$M(I_{C}) = \max \{ d , 1+ \binom{d-1}{2}-g(C)\}.$$
\end{Thm}


\section{Degree complexity of smooth irreducible surfaces in $\p^{4}$}\label{section_4}
Let $S$ be a non-degenerate smooth irreducible surface of degree $d$ and arithmetic genus $\rho_{a}(S)
$ in $\p^4$ and let $I_S$ be the defining ideal of $S$ in $R=k[x_0,\ldots,x_4]$.
In this section, we study the scheme structure of
$$
Z_i(S):=\Proj(\bar R/K_i(I_S)), \,\text{ where } \bar  R=k[x_1,x_2, x_3, x_4]
$$
arising from a generic projection in order to get a geometric
interpretation of the degree-complexity $M(I_S)$ of $S$ in $\p^4$ with
respect to the degree lexicographic order.

We recall without proof the standard facts concerning generic projections of surfaces in $\p^4$ to $\p^3$.

Let $S\subset \p^4$ be a non-degenerate smooth irreducible surface of degree $d$ and arithmetic genus $\rho_{a}(S)$
and $\pi:S\rightarrow \pi(S)\subset \p^3$ be a generic projection.

(a) The singular locus of $\pi(S)$ is a curve $Y_1(S)$ with only singularities a number $t$ of ordinary triple points
with transverse tangent directions. The inverse image $\pi^{-1}(Y_1(S))$ is a curve with only singularities $3t$ nodes,
$3$ nodes above each triple point of $Y_1(S)$ (see \cite{P}).
This implies (using Proposition~\ref{defset}) that the ideals $K_j(I_S)$ have finite colength if $j>2$.
This fact is used in the proofs of Propostion~\ref{thm4.2} and Theorem~\ref{max3}.

(b) If a smooth surface $S\subset \mathbb P^4$ is contained in a quadric hypersurface
then there are no ordinary triple points in $Y_1(S)$.
This implies that the double curve $Y_1(S)$ is smooth by (a).

(c) The double curve $Y_1(S)$ is irreducible unless $S$ is a projected Veronese surface in $\p^4$ (see \cite {MP}).

(d) The reduced induced scheme structure on $Y_1(S)$ is defined by the first Fitting ideal of the $\mathcal O_{\p^3}$-module
$\pi_{*}\mathcal O_{S}$ (see \cite {MP}).

(e) The degree of $Y_1(S)$ is $\binom{d-1}{2}-g(S\cap H)$ where $S\cap H$ is a general hyperplane section
and the number of apparent triple points $t$ is given in \cite{Lb} by
$$
t=\binom{d-1}{3}-g(S\cap H)(d-3)+2\chi(\mathcal O_S)-2.
$$

The following lemma shows that the Hilbert function of $I_{S}$ can be obtained from those of partial elimination ideals $K_i(I_S)$.
\begin{lem}\label{lem3.4}
Let $S \subset \p^{4}$ be a smooth surface with the
defining ideal $I_{S}$ in $R=k[x_0, x_1, \ldots , x_4]$. Consider a projection
$\pi_q:S\longrightarrow \p^3$ from a general point $q=[1,0,0,0,0] \notin S$.
Then,
\begin{equation*}\label{4}
H(R/I_S,m) = \sum_{i \geq 0} H(\bar{R}/K_i(I_{S}),m-i).
\end{equation*}
In particular,
$$
P_S(z) = P_{Z_0(S)}(z)+P_{Z_1(S)}(z-1)+P_{Z_2(S)}(z-2).
$$
\end{lem}
\begin{proof}
The equality on Hilbert functions basically comes from the following combinatorial identity
$$
\binom{m+d}{d} = \sum_{i=0} ^{d}\binom{m-1+d-i}{d-i}.
$$

For a smooth surface $S \subset \p^4$, $Z_i(S)=\emptyset$ for $i \geq 3$ by the (dimension +2)-secant lemma
(see \cite {R}) and so $\bar{R}/K_i(I_{S})$ is Artinian. Thus
$P_{Z_i(S)}(z) = 0 ~\mbox{for $i \geq 3$}$
(see \cite[Lemma 3.4]{A} for details).
\end{proof}

The following theorem says that the first partial elimination ideal
$K_1(I_S)$ gives the reduced induced scheme structure on the double
curve $Y_1(S)$ in $\p^3$ (i.e., $I_{Z_1(S)}=I_{Y_1(S)}$).
\begin{Thm}\label{thm4.1}
Suppose that $S$ is a reduced irreducible surface in $\p^4$. Then,
\begin{itemize}
 \item[(a)] the first partial elimination ideal
$K_1(I_S)$ is a saturated ideal, so we have $K_1(I_S)=I_{Z_1(S)}$;
 \item[(b)] if $S$ is a smooth surface contained in a
quadric hypersurface, then $K_1(I_S)=I_{Y_1(S)}$, which implies that $K_1(I_S)$ is a reduced ideal.
\end{itemize}
\end{Thm}
\begin{proof}
(a) Assume that $S$ is a reduced irreducible surface in $\p^4$ of degree $d$.
Take a general point $q\in \mathbb P^4$; we may assume $q=[1,0,\ldots,0]$.
Then the generic projection of $S$ into $\mathbb P^3$ from the point $q$ is defined
by a single polynomial $F\in k[x_1,x_2,x_3, x_4]$ of degree $d$ and $K_0(I_S)=(F)$, which is a reduced ideal.

Let $\bar {\mathcal M}=(x_1,x_2,x_3, x_4)$ be the irrelevant maximal ideal of
$\bar{R}=k[x_1,x_2,x_3, x_4]$. By the definition of saturated ideal,
$K_1(I_S)$ is saturated if and only if
$$(K_1(I_S):\bar {\mathcal
M})=K_1(I_S).$$ Hence it is enough to show that
\[(K_1(I_S):\bar {\mathcal M})/K_1(I_S)=0.\]
For the proof, consider the Koszul complex
\[\cdots \rightarrow \mathcal K_m^{-p-1}\rightarrow \mathcal K_m^{-p}
\rightarrow \mathcal K_m^{-p+1}\rightarrow \cdots,\] where
$\mathcal K_m^{-p}=\bigwedge^{p}\bar {\mathcal M}\bigotimes
K_0(I_S)_{m-p}$. From Corollary 6.7 in \cite{G}, the $\bar
R$-module $(K_1(I_S):\bar {\mathcal M})_d/K_1(I_S)_d$ injects into
$H^{-1}(\mathcal K^{\bullet}_{d+3})$ for each $d$. Note that
\[H^{-1}(\mathcal K^{\bullet}_{d+3})=H(\bigwedge^{1}\bar {\mathcal M}\bigotimes
K_0(I_S)_{d+2})=\Tor_1^{\bar R}(\bar R/\bar {\mathcal
M},K_0(I_S))_{d+3}.\] Since the ideal $K_0(I_S)$ is generated by a
single polynomial $F$, we have that \[\Tor_1^{\bar R}(\bar
R/\bar {\mathcal M},K_0(I_S))=0.\] This proves that
$(K_1(I_S):\bar
{\mathcal M})/K_1(I_S)=0$, as we wished.\\

(b) Since $S$ is contained in a quadric hypersurface and the center of projection is
outside a quadric, we have a surjection $\varphi: \bar R(-1)\oplus \bar R\rightarrow R/I_S$
as a $\bar R$-module homomorphism with the following diagram:
\begin{equation*}
\begin{array}{ccccccccccccccccccccccc}
&&0&&0&&0&&\\[1ex]
&&\downarrow &&\downarrow&&\downarrow&&\\[1ex]
0&\longrightarrow& K_0(I_S) & \longrightarrow & \bar{R}& \longrightarrow &
\bar{R}/K_0(I_S)&\longrightarrow & 0 \\[1ex]
&&\downarrow && \downarrow && \,\,\,\downarrow & \\[1ex]
0&\longrightarrow& \widetilde{K}_{1}(I_S) & \longrightarrow & \bar{R}\oplus
\bar{R}(-1)& \st{{\varphi}}  & R/I_S
&\longrightarrow & 0\\[1ex]
&&\downarrow &&\downarrow && \downarrow &&\\[1ex]
0&\longrightarrow& K_1(I_S)(-1) & \longrightarrow & \bar{R}(-1)&\longrightarrow& \bar{R}/K_1(I_S)(-1)\,
&\longrightarrow &0\\[1ex]
&&\downarrow &&\downarrow&&\downarrow&&\\[1ex]
&&0&&0&&0&&
\end{array}
\end{equation*}
where $\widetilde{K}_{1}(I_S)=\{f \in I_S\,\mid\, d_0(f)\leq 1\}$ is an $\bar R$-module.
Let $\mathcal O_{Z_1(S)}$ be the sheafification of $\bar{R}/K_{1}(I_S)$.
By sheafifying the rightmost vertical sequence, we have
\begin{equation}\label{exact seq}
0 \longrightarrow \mathcal{O}_{\pi(S)} \longrightarrow
\pi_{*}\mathcal{O}_{S} \longrightarrow
\mathcal O_{Z_1(S)}(-1) \longrightarrow 0.
\end{equation}

Let $\mathcal I_{Z_1(S)}=\mathcal{K}_{1}(I_S)$ be the sheafification of the ideal
$K_{1}(I_S)$. In \cite[(3.4.1), p. 302]{KLU}, S. Kleiman, J. Lipman and B. Ulrich proved that
\begin{equation*}
\mathcal I_{Y_1(S)}
= \mathrm{Fitt}_{1}^{\p^3}(\pi_{*}\mathcal{O}_S)
 =\mathrm{Fitt}_{0}^{\p^3}(\pi_{*}\mathcal{O}_S/\mathcal{O}_{\pi(S)})
 =\mathrm{Ann}_{\p^3}(\mathcal O_{Z_1(S)}(-1)),
\end{equation*}
and this defines {\it the reduced scheme structure} on $Y_1(S)$ (see \cite[p. 3] {MP}).

On the other hand, from the sequence~(\ref{exact seq}), we have
$$\mathcal{I}_{Y_1(S)} = \mathrm{Ann}_{\p^3}(\mathcal O_{Z_1(S)}(-1)) = \mathcal{K}_{1}(I_S)=\mathcal I_{Z_1(S)}.$$
Then it follows from (a) that
$$I_{Z_1(S)}=K_{1}(I_S)^{\mathrm{sat}}=K_{1}(I_S)=I_{Y_{1}(S)}.$$
Since $I_{Y_1(S)}$ is a reduced ideal, we conclude that $I_{Z_1(S)}=K_{1}(I_S)$ is also a reduced ideal.
\end{proof}

If $S\subset \p^4$ is contained in a quadric hypersurface, then by Theorem~ \ref{thm4.1},
$K_1(I_S)$ is saturated and reduced. So, it defines the reduced scheme structure on $Y_1(S)$.
Note also that the double curve $Y_1(S)$ is smooth (see the standard fact (b) in the beginning of this section).
We use this fact to prove the following theorem.

\begin{Thm}\label{max3}
Let $S$ be a smooth irreducible surface of degree $d$ lying on a
quadric hypersurface in $\p^4$. Let $Y_1(S)$ be the double curve of genus $g(Y_1(S))$ defined by
a generic projection $\pi$ of $S$ to $\p^3$. Then, we have the following;
\begin{itemize}
\item[(a)]
$M(I_S)= \max \{ d, 1+\deg Y_1(S), 2+\binom{\deg Y_1(S)-1}{2}-g(Y_1(S))$\};
\item[(b)]
$M(I_S)$ can be obtained at one of monomials
$$
x_1^d, \,\,x_0x_2^{\deg Y_1(S)},\,\, x_0 x_1 x_3^{\binom{\deg Y_1(S)-1}{2}-g(Y_1(S))}.
$$
\end{itemize}
\end{Thm}
\begin{proof}
Note that by Corollary~\ref{cor:305},
$$
M(I_S)=\underset{0\leq i \leq \beta}{\max} \{\reg(\Gin(K_i(I_S)))+i \},
$$
where $\beta=\min\{j\mid K_j(I_S)= \bar R\}$.
Since $S$ is contained in a quadric hypersurface, $\Gin(I_S)$ contains the monomial $x_0^2$.
This means that $\Gin(K_2(I_S))= \bar R $.
On the other hand, $\Gin(K_0(I_S))=(x_1^d)$ by the Borel fixed property because $\pi(S)$ is
a hypersurface of degree $d$ in $\p^3$ and $I_{\pi(S)}=K_0(I_S)$.
Thus $\Gin(I_S)$ is of the form
$$
(x_0 ^2, x_0  g_1, x_0 g_2, \ldots , x_0 g_m, x_1^d).
$$
Note that $g_1, \ldots g_m$ are monomial generators of $\Gin(K_1(I_S))=\Gin(I_{Y_1(S)})$
by Proposition~\ref{thm2}.

Therefore, by Theorem~\ref{curveM},
$$
\reg(\Gin(K_1(I_S)))=\max\{ \deg Y_1(S), 1+\binom{\deg
Y_1(S)-1}{2}-g(Y_1(S)) \}$$ and consequently,
$$
M(I_S)=\max \{ d, 1+\deg Y_1(S), 2+\binom{\deg Y_1(S)-1}{2}-g(Y_1(S))\}.
$$
For a proof of (b), consider $\Gin(K_1(I_S))=\langle g_1, g_2, \ldots, g_m \rangle$
in (a).
Note that the double curve $Y_1(S)$ is smooth in $\p^3$. By the similar argument used in (a),
$\Gin(K_1(I_S))$ contains $x_2^{\deg(Y_1(S))}$ because the image of $Y_1(S)$ under a generic projection
to $\p^2$ is a plane curve of degree $\deg(Y_1(S))$. Finally, consider all monomial generators of the form
$x_1\cdot h_j(x_2, x_3, x_4)$ in $\{ g_1, g_2, \ldots, g_m \}$. Then, $\{h_j(x_2, x_3, x_4)\mid 1\le j \le m \}$
is a minimal generating set of $\Gin(K_1(I_{Y_1(S)}))$ by Proposition~\ref{thm2}. Recall that $K_1(I_{Y_1(S)})$
defines $\binom{\deg Y_1(S)-1}{2}-g(Y_1(S))$ distinct nodes in $\p^2$.
So, $\Gin(K_1(I_{Y_1(S)}))$ should contain the monomial $x_3^{\binom{\deg Y_1(S)-1}{2}-g(Y_1(S))}$
(see also \cite[Corollary 5.3]{CS}).
Therefore, $\Gin(I_S)$ contains monomials $x_1^d, x_0x_2^{\deg(Y_1(S))}$ and
$x_0x_1x_3\,^{\binom{\deg Y_1(S)-1}{2}-g(Y_1(S))}$.
\end{proof}

\begin{remk}
In the proof of Theorem~\ref{max3}, we showed that if a smooth irreducible surface $S$
is contained in a quadric hypersurface
then $M(I_S)$ is determined by two partial elimination ideals $K_0(I_S)$ and $K_1(I_S)$
since $K_i(I_S)=\bar R$ for all $i\geq 2$.
\end{remk}

The following theorem shows that if $d\geq 6$ then $M(I_S)$ is determined by the degree complexity
of the first partial elimination ideal $K_1(I_S)$.

\begin{prop}\label{mainthm2}
Let $S$ be a smooth irreducible surface of degree $d$ in $\p^4$. Suppose that $S$ is contained
in a quadric hypersurface. Then
\[
M(I_{S})=
\begin{cases}
3  &\text{if $S$ is a rational normal scroll with $d=3$}\\
4  &\text{if $S$ is a complete intersection of (2,2)-type}\\
5  &\text{if $S$ is a Castelnuovo surface with $d=5$}\\
2+\binom{\deg Y_1(S)-1}{2}-g(Y_{1}(S))  & \,\,\text {for}\,\,d\ge 6
\end{cases}
\]
where $Y_1(S)\subset \p^3$ is a double curve of degree $\binom{d-1}{2}-g(S\cap H)$
under a generic projection of $S$ to $\p^3$.
\end{prop}
\begin{proof}
Since $K_2(I_S)=\bar{R}$, Theorem ~\ref{max3} implies that
\begin{equation*}
M(I_S)= \max \{ d, 1+\deg Y_1(S), 2+\binom{\deg
Y_1(S)-1}{2}-g(Y_1(S))\}.
\end{equation*}

If $\deg Y_{1}(S)\geq 5$ then by the genus bound,
$$1 + \deg Y_1(S) \leq 2 +\binom{\deg Y_1(S)-1}{2}-g(Y_1(S)).$$

We claim that if $d\geq 6$, then $d\leq 1 + \deg Y_1(S)$.
Notice that from our claim, we have the degree complexity of a surface lying on
a quadric hypersurface in $\p^4$  for $d\geq 6$ as follows;
$$M(I_S)=2+\binom{\deg Y_1(S)-1}{2}-g(Y_1(S)).$$
Note again that
\[
g(S\cap H) \leq \pi(d,3) =
\begin{cases}
(\frac{d}{2}-1)^{2}             &\text{if $d$ is even;}\\
(\frac{d-1}{2})(\frac{d-3}{2})  &\text{if $d$ is odd.}
\end{cases}
\]
Then we can show that $\pi(d,3) \leq \binom{d-1}{2}-d+1$ if $d=\deg(S\cap H)\geq 6$.
Thus, if $d\geq 6$ then
\begin{equation*}
d\leq 1 + \binom{d-1}{2} - g(S\cap H)=1+\deg Y_1(S).
\end{equation*}
So, our claim is proved and only three cases of $d=3,4,5$ are remained.

\medskip

Case 1: If $\deg S=3$ then $S$ is a rational normal scroll with
$g(S\cap H)=0$ and the double curve $Y_1(S)$ is a line. So, by simple computation, $M(I_S)=3$.

\medskip

Case 2:
If $\deg S=4$ then $S$ is a complete intersection of (2,2)-type with $g(S \cap H)=1$ and
the double curve $Y_1(S)$ is a plane conic of $\deg Y_1(S)=2$. So, by simple computation,
$M(I_S)=4$.
\medskip

Case 3:
If $\deg S=5$ then $S$ is a Castelnuovo surface with $g(S\cap H)=2$ and the double curve
$Y_1(S)\subset \p^3$ is a smooth elliptic curve of degree $4$.
In this case, we can also compute
$$M(I_S)=5=\deg S > 2+\binom{\deg Y_1(S)-1}{2}-g(Y_1(S))=4.$$
\end{proof}

\begin{prop}\label{thm4.2}
Let $S$ be a smooth irreducible surface of degree $d$ and arithmetic
genus $\rho_{a}(S)$ in $\p^4$. Let $Y_i(S)$ be the multiple locus
defined by a generic projection of $S$ to $\p^3$ for $i\ge 0$.
Assume that $S$ is contained in a quadric hypersurface.
Then, the following identity holds;
$$g(Y_{1}(S))=\binom{d-1}{3}-\binom{d-1}{2}+g(S\cap H)-\rho_{a}(S)+1.$$
\end{prop}
\begin{proof}
Let $P_S(z)$ be the Hilbert polynomial of a smooth irreducible
surface of degree $d$ and arithmetic genus $\rho_{a}(S) $. Since $Y_2(S)=\emptyset$, $P_{Y_2(S)}(z)=0$
and, by Lemma~\ref{lem3.4},
$$
P_S(z)= P_{Y_{0}(S)}(z)+ P_{Y_{1}(S)}(z-1).
$$
Plugging $z=0$, $P_S(0)=\rho_a(S)+1, P_{Y_0(S)}(0)= \binom{d-1}{3}+1,$ and
$$
P_{Y_1(S)}(-1)=-\deg Y_1(S)+1-g(Y_1(S))= -\binom{d-1}{2}+g(S\cap H)+1-g(Y_1(S)).$$
Therefore, we have the following identity:
$$g(Y_{1}(S))=\binom{d-1}{3}-\binom{d-1}{2}+g(S\cap H)-\rho_{a}(S)+1.$$
\end{proof}

\begin{remk}\label{genusdouble}
By Proposition ~\ref{thm4.2}, when $d\ge 6$, $M(I_S)$ can be expressed with only three invariants of $S$:
its degree, sectional genus, and arithmetic genus, as follows:
$$M(I_S)=\binom{\binom{d-1}{2}-g(S \cap H)-1}{2}- \binom{d-1}{3}+\binom{d-1}{2}-g(S\cap H)+\rho_a(S)+1.$$ \qed
\end{remk}

In order to compute $M(I_S)$ in terms of degrees of defining equations as A. Conca and J. Sidman did in \cite {CS},
we need the following remark. This shows that a smooth surface in $\p^4$ has a nice algebraic structure
when it is contained in a quadric hypersurface.
\begin{remk}\label{CM surface}
Let $S$ be a locally Cohen-Macaulay surface lying on a quadric
hypersurface $Q$ in $\mathbb P^4$. Then $S$ satisfies one of
following conditions (see \cite[Theorem 2.1]{K});
 \begin{itemize}
  \item[(a)] $S$ is a complete intersection of $(2,\alpha)$-type.
   \begin{itemize}
    \item[(i)] $I_S=(Q, F)$, where $F$ is a polynomial of degree $\alpha$.
    \item[(ii)] $\reg(S)=\alpha+1$.
   \end{itemize}
   \item[(b)] $S$ is projectively Cohen-Macaualy of degree $2\alpha-1$.
    \begin{itemize}
    \item[(i)] $I_S=(Q, F_1, F_2)$, where $F_1$ and $F_2$ are polynomials of degree $\alpha$.
    \item[(ii)] $\reg(S)=\alpha$.
   \end{itemize}
 \end{itemize}
 \end{remk}

From the above Remark ~\ref{CM surface}, we can compute $g(S \cap H)$ and $\rho_a(S)$
in terms of the degree of defining equations of $S$ by finding the Hilbert polynomial
of $S$ in two ways. Therefore, we have the following Theorem.

\begin{Thm}{\label{cor:01}}
Let $S\subset \p^4$ be a smooth irreducible surface of degree $d$ and arithmetic
genus $\rho_{a}(S)$, which is contained in a quadric hypersurface.
\begin{itemize}
\item[(a)] Suppose $S$ is of degree $2\alpha, \alpha\geq 3$. Then,
$$M(I_S) = \frac{1}{2}(\alpha^4-4\alpha^3+5\alpha^2-2\alpha+4).$$
\item[(b)] Suppose $S$ is of degree $2\alpha-1, \alpha\geq 4$. Then
$$M(I_S) = \frac{1}{2}(\alpha^4-6\alpha^3+13\alpha^2-12\alpha+8).$$
\end{itemize}
\end{Thm}
\begin{proof} For a proof of (a), by Koszul complex we have the minimal free resolution
of the defining ideal $I_S$ as follows:
$$
0 \longrightarrow R(-\alpha-2) \longrightarrow R(-2)\oplus
R(-\alpha) \longrightarrow I_S\longrightarrow 0,
$$
Hence the Hilbert function of $R/I_S$ is given by
\begin{align*}
H(R/I_S,m)=&\alpha m^2 + (- \alpha^2 +3 \alpha )m +
\frac{1}{6}\alpha(2\alpha^2-9\alpha+13)\\[1ex]
=& \frac{2\alpha}{2}m^2+\left(\alpha+1-g(S\cap H)\right )m
+\rho_a(S)+1.
\end{align*}
Hence $g(S \cap H) = (\alpha-1)^{2}$ and
$\rho_a(S)=\frac{1}{6}\alpha(2\alpha^2-9\alpha+13)-1.$\\
If $Y_1(S)$ is the double curve of $S$ then
$$\deg Y_1(S)=\binom{2\alpha-1}{2}-g(S \cap H)=\alpha(\alpha-1).$$
By Remark~\ref{genusdouble},
$$g(Y_1(S))=\binom{2\alpha-1}{3}-\binom{2\alpha-1}{2}+g(S\cap H)-\rho_a(S)+1.$$

Thus we conclude that
\begin{equation*}
\begin{split}
M(I_S) = & 2+\binom{\alpha(\alpha-1)-1}{2}-g(Y_1(S))\\[1ex]
           = & \binom{\alpha(\alpha-1)-1}{2}-\binom{2\alpha-1}{3}
+\binom{2\alpha-1}{2}-(\alpha-1)^2 +\rho_a (S) + 1\\[1ex]
           =&\frac{1}{2}(\alpha^4-4\alpha^3+5\alpha^2-2\alpha+4).
\end{split}
\end{equation*}
For a proof of (b), let $S$ be a smooth surface of degree $2\alpha-1$ lying on a quadric
hypersurface in $\p^4$.
Note that $S$ is arithmetically Cohen-Macaulay of codimension 2. By the
Hilbert-Burch Theorem \cite{E} we have the minimal free resolution
of the defining ideal $I_S$ as follows:
\begin{equation*}
0\longrightarrow R(-\alpha-1)^2 \stackrel{ \left(\begin{array}{lll}
L_1 & L_2\\
L_3 & L_4\\
F_5 & F_6\\
\end{array}
\right ) }{\longrightarrow} R(-2)\oplus R(-\alpha)^2 \longrightarrow
I_S\longrightarrow0,
\end{equation*}
where $L_1, L_2, L_3, L_4$ are linear forms and $F_5, F_6$ are forms
of degree $\alpha-1$. Hence the Hilbert function of $R/I_S$ is given
by

\begin{align*}
H(R/I_S,m)=&\frac{1}{2}(2\alpha-1)m^2 + \left(4\alpha-\alpha^2-\frac{3}{2}\right )m+\frac{1}{3}\alpha^3-2\alpha+\frac{11}{3}\alpha-1\\[1ex]
=& \frac{(2\alpha-1)}{2}m^2+\left(\frac{2\alpha-1}{2}+1 -g(S\cap H)\right )m +\rho_a(S)+1.
\end{align*}

Hence we have that $g(S \cap H) =
2\displaystyle\binom{\alpha-1}{2}$ and $\rho_a(S)=2\displaystyle
\binom{\alpha-1}{3}$.

If $Y_1(S)$ be the double curve of $S$ then
$$\deg Y_1(S)=\binom{2\alpha-2}{2}-g(S \cap H)=\binom{2\alpha-2}{2}-2\binom{\alpha-1}{2}.$$

On the other hand, we have
\begin{align*}
g(Y_1(S))=&\,\binom{2\alpha-2}{3}-\binom{2\alpha-2}{2}+g(S\cap H)-\rho_a(S)+1\\[1ex]
                    =&\,(\alpha-2)(\alpha^2-3\alpha+1)
\end{align*}
 and thus we conclude that
\begin{equation*}
\begin{split}
M(I_S) = &\, 2+\binom{\deg Y_1(S)-1}{2}-g(Y_1(S))\\[1ex]
           =&\,\frac{1}{2}(\alpha^4-6\alpha^3+13\alpha^2-12\alpha+8).
\end{split}
\end{equation*}
\end{proof}

\begin{Ex}[Macaulay 2]~\label{Exam:01}
We give some examples of $\Gin(I_S)$ and $M(I_S)$ computed by using {\it Macaulay 2}.
\begin{itemize}
 \item[(a)] Let $S$ be a rational normal scroll in $\p^4$ whose defining ideal is
$$I_S= ( x_{0} x_{3} - x_{1} x_{2} , x_{0} x_{1} - x_{3} x_{4} , x_{0}^{2} - x_{2} x_{4}).$$
Using Macaulay 2, we can compute the generic initial ideal of $I_S$
with respect to GLex:
$$\Gin(I_S)=(x_0^2, x_0x_1, x_0x_2, {\bf x_1^3}).$$
Thus $\reg(\Gin_{\lex}(K_{0}))=3$ and $\reg(\Gin_{\lex}(K_{1}))=1$.
Therefore,
$$
M(I_S) = \deg S = 3.
$$
 \item[(b)] Let $S$ be a complete intersection of $(2,2)$-type in $\p^4$. Then,
 $$\Gin(I_S)=(x_0^2, x_0x_1, {\bf x_1^4}, x_0x_2^2).$$
 Hence, we see $M(I_S)=\deg S=4$.
 \item[(c)] Let $S$ be a Castelnuovo surface of degree 5 in $\p^4$. Then, we can compute
 $$\Gin(I_S)=(x_0^2, x_0x_1^2, {\bf x_1^5}, x_0x_1x_2, x_0x_2^4, x_0x_1x_3^2).$$
 Hence, we see $M(I_S)=\deg S=5$.
 \item[(d)] Let $S$ be a complete intersection of $(2,3)$-type in $\p^4$. Then, we see that $M(I_S)=8$ from Theorem~\ref{cor:01}. On the other hand, we can compute the generic initial ideal:
 $$
 \Gin(I_S)= (x_0^2, x_0x_1^2, x_1^6, x_0x_1x_2^2, x_0x_2^6, x_0x_1x_2x_3^2, {\bf x_0x_1x_3^6}, x_0x_1x_2x_3x_4^2,  x_0x_1x_2x_4^4).
 $$
 This also shows $M(I_S)=8$.
 \item[(e)] Let $S$ be a smooth surface of degree 7 lying on a quadric which is not a complete intersection in $\p^4$. Then, the minimal resolution of $I_S$ is given by Hilbert-Burch Theorem and thus we have
 $$I_S=(L_1L_4-L_2L_3, L_1F_5-L_2F_6, L_3F_5-L_4F_6),$$
 where $L_i$ is a linear form and $F_5, F_6$ are forms of degree 3. This is the case of $\alpha=4$ in Theorem~\ref{cor:01} and we see $M(I_S)=20$. This can also be obtained by the computation of generic initial ideal of $I_S$ using {\it Macaulay 2}:
\begin{align*}
\Gin(I_S)=&(x_0^2,\, x_0x_1^3,\, x_1^7,\, x_0x_1^2x_2,\, x_0x_1x_2^4,\, x_0x_2^9,\, x_0x_1^2x_3^2,\,
            x_0x_1x_2^3x_3^2,\, x_0x_1x_2^2x_3^5,\,\\
            & x_0x_1x_2x_3^8,\, {\bf x_0x_1x_3^{18}},\, x_0x_1x_2^2x_3^4x_4,\, x_0x_1^2x_3x_4^2,\, x_0x_1x_2^3x_3x_4^2,\,
            x_0x_1x_2^2x_3^3x_4^2,\, \\
            & x_0x_1x_2x_3^7x_4^2,\, x_0x_1x_2^3x_4^3,\, x_0x_1^2x_4^4,\,
             x_0x_1x_2^2x_3^2x_4^4,\, x_0x_1x_2x_3^6x_4^4,\, x_0x_1x_2^2x_3x_4^5,\, \\
            & x_0x_1x_2x_3^5x_4^6,\, x_0x_1x_2^2x_4^7,\, x_0x_1x_2x_3^4x_4^8,\, x_0x_1x_2x_3^3x_4^{10},\, x_0x_1x_2x_3^2x_4^{12},\, \\
            &x_0x_1x_2x_3x_4^{14},\, x_0x_1x_2x_4^{16})
\end{align*}

  \item[(f)] Let $S$ be a complete intersection of $(2,4)$-type in $\p^4$. Then, we see that $M(I_S)=38$ from Theorem~\ref{cor:01}. This can be given by the computation of generic initial ideal of $I_S$:
  \begin{align*}
\Gin(I_S)&= (x_0^2,\, x_0x_1^3,\, x_1^8,\, x_0x_1^2x_2^2,\, x_0x_1x_2^6,\, x_0x_2^{12},\, x_0x_1^2x_2x_3^2,\, x_0x_1x_2^5x_3^2,\,  \\
            & x_0x_1^2x_3^5,\, x_0x_1x_2^4x_3^5,\, x_0x_1x_2^3x_3^7,\, x_0x_1x_2^2x_3^{11} ,\, x_0x_1x_2x_3^{17},\, {\bf x_0x_1x_3^{36}},\, \\
            & x_0x_1^2x_3^4x_4,\, x_0x_1x_2^4x_3^4x_4,\, x_0x_1x_2^3x_3^6x_4,\, x_0x_1x_2^2x_3^{10}x_4,\, x_0x_1^2x_2x_3x_4^2,\,\\
            & x_0x_1x_2^5x_3x_4^2,\, x_0x_1^2x_3^3x_4^2,\, x_0x_1x_2^4x_3^3x_4^2,\, x_0x_1x_2^2x_3^9x_4^2 ,\,x_0x_1x_2x_3^{16}x_4^2,\,\\
           & x_0x_1^2x_2x_4^3 ,\,x_0x_1x_2^5x_4^3,\,x_0x_1x_2^4x_3^2x_4^3 ,\, x_0x_1x_2^3x_3^5x_4^3 ,\, x_0x_1^2x_3^2x_4^4,\,\\
           & x_0x_1x_2^3x_3^4x_4^4 ,\, x_0x_1x_2^2x_3^8x_4^4 ,\, x_0x_1x_2x_3^{15}x_4^4,\, x_0x_1^2x_3x_4^5,\, x_0x_1x_2^4x_3x_4^5,\, \\
           &x_0x_1x_2^3x_3^3x_4^5,\,  x_0x_1x_2^2x_3^7x_4^5,\, x_0x_1x_2^4x_4^6,\, x_0x_1x_2x_3^{14}x_4^6,\, x_0x_1^2x_4^7,\,\\
           & x_0x_1x_2^3x_3^2x_4^7,\, x_0x_1x_2^2x_3^6x_4^7,\, x_0x_1x_2^3x_3x_4^8,\, x_0x_1x_2^2x_3^5x_4^8,\, x_0x_1x_2x_3^{13}x_4^8,\,\\
           &  x_0x_1x_2^3x_4^9,\,x_0x_1x_2^2x_3^4x_4^{10},\, x_0x_1x_2x_3^{12}x_4^{10},\, x_0x_1x_2^2x_3^3x_4^{11},\,x_0x_1x_2x_3^{11}x_4^{12},\, \\
           & x_0x_1x_2^2x_3^2x_4^{13},\, x_0x_1x_2^2x_3x_4^{14},\,x_0x_1x_2x_3^{10}x_4^{14},\, x_0x_1x_2^2x_4^{16},\, x_0x_1x_2x_3^9x_4^{16},\,\\
           & x_0x_1x_2x_3^8x_4^{18},\, x_0x_1x_2x_3^7x_4^{20},\, x_0x_1x_2x_3^6x_4^{22},\,x_0x_1x_2x_3^5x_4^{24},\, x_0x_1x_2x_3^4x_4^{26},\,\\
           &  x_0x_1x_2x_3^3x_4^{28},\,x_0x_1x_2x_3^2x_4^{30},\, x_0x_1x_2x_3x_4^{32},\, x_0x_1x_2x_4^{34}).
\end{align*}

\end{itemize}
\end{Ex}

Even though we cannot compute the generic initial ideals for the cases $\alpha\geq 5$ by using computer algebra systems, we know the degree-complexity of smooth surfaces lying on a quadric by theoretical computations.  We give the following tables:\\
\vskip.5pc

{\hrule\hrule\vskip .5pc
\begin{flushleft}
{\bf Table 1}  The complete intersection $S$ of $(2,\alpha)$-type in $\p^4$
\end{flushleft}
\vskip .2pc
\hrule\hrule
{
\begin{flushleft}
 \begin{tabular}{|c|c|c|c|c|c|c|c|c|c|c|c|c|c|c|c|c|c|c|c|c|c|c|c|c|c}
$\alpha$  & 5 & 6   & 7     & 8     & 9     & 10     & 20    & 50   & 100   \\[1ex]
\hline
$M(I_S)$  & 122 & 302 & 632 & 1178 & 2018 & 3242 & 64982  & 2881202 & 48024902  \\[1ex]
\hline
    $m(I_S)$ & 6 & 7   & 8     & 9     & 10      & 11      &  21     & 51      &  101   \\[1ex]
 \end{tabular}

\end{flushleft}
}
\hrule\hrule\vskip .5pc
}

{\hrule\hrule\vskip .5pc
\begin{center}
{\bf Table 2}  The smooth surface $S\subset \p^4$ of degree $(2\alpha-1)$ lying on a quadric.
\end{center}
\vskip .2pc
\hrule\hrule
{
\begin{center}
 \begin{tabular}{|c|c|c|c|c|c|c|c|c|c|c|c|c|c|c|c|c|c|c|c|c|c|c|c|c|c}
$\alpha$  & 5 & 6   & 7     & 8     & 9     & 10     & 20    & 50   & 100   \\[1ex]
\hline
$M(I_S)$  & 74 & 202 & 452 & 884 & 1570 & 2594 & 58484  & 2765954 & 47064404  \\[1ex]
\hline
    $m(I_S)$ & 5 & 6 & 7   & 8     & 9     & 10      & 20      &  50     & 100 \\[1ex]
 \end{tabular}

\end{center}
}
\hrule\hrule\vskip .5pc
}

\bigskip
\begin{RQ}\label{mainthm3}
Let $S$ be a non-degenerate smooth surface of degree $d$ and
arithmetic genus $\rho_{a}(S)$, not necessarily contained in a quadric hypersurface in $\p^{4}$. 
Our question is: What can be the degree complexity $M(I_S)$ of $S$? 
It is expected that $K_1(I_S)$ and $K_2(I_S)$ are reduced ideals and the degree-complexity $M(I_S)$ is given by
\[
\begin{array}{lll}
M(I_{S})&=& \max
\begin{cases}
\deg(S)  &\text{ }\\
\reg(\Gin_{\lex}(K_1(I_{S})))+1 &\\
\reg(\Gin_{\lex}(K_2(I_{S})))+2
\end{cases}\\
&&\\
&=& \max
\begin{cases}
d  &\text{ }\\
M(I_{Y_1(S)})+1 &\text{ }\\
t+2.
\end{cases}
\end{array}
\]
Note that $t$ is the number of apparent triple points of $S\subset \p^4$
and $Y_1(S)$ is the double curve (possibly singular with ordinary double points)
under a generic projection.
\end{RQ}
\qed

\bibliographystyle{amsalpha}

\end{document}